\documentclass[10pt, conference, compsocconf]{IEEEtran}
\usepackage[]{amsmath,amssymb,amsfonts,latexsym,amsthm,enumerate,fullpage}
\newcommand{\remove}[1]{}

\newcommand{\calB}{{\cal B}}

\def\F{\mathbb{F}}
\def\Q{\mathbb{Q}}
\def\R{\mathbb{R}}

\def\P{\mathbb{P}}

\newtheorem{thm}{Theorem}[section]
\newtheorem{definition}[thm]{Definition}

\newtheorem{question}[thm]{Question}
\newtheorem{cor}[thm]{Corollary}

\newtheorem{lemma}[thm]{Lemma}
\newtheorem{proposition}[thm]{Proposition}

\begin{document}
\title{Ramanujan Complexes and bounded degree topological expanders}
\author{\IEEEauthorblockN{Tali Kaufman}
\IEEEauthorblockA{Bar-Ilan University\\
Israel\\
kaufmant@mit.edu}
\and
\IEEEauthorblockN{David Kazhdan}
\IEEEauthorblockA{Hebrew University\\
Israel\\
kazhdan.david@gmail.com}
\and
\IEEEauthorblockN{Alexander Lubotzky}
\IEEEauthorblockA{Hebrew University\\
Israel\\
alexlub@math.huji.ac.il}
}
%\author{Tali Kaufman \thanks{Bar-Ilan University, ISRAEL. Email: \texttt{kaufmant@mit.edu}.
%Research supported in part by the Alon Fellowship, IRG, ERC and BSF.}
% \and David Kazhdan \thanks{Hebrew University, ISRAEL. Email: \texttt{kazhdan.david@gmail.com}.
%Research supported in part by the NSF, BSF and ERC.}
% \and Alexander Lubotzky \thanks{Hebrew University, ISRAEL. Email: \texttt{alexlub@math.huji.ac.il}.
%Research supported in part by the ERC, NSF and ISF.}}
%%\author{Tali Kaufman, David Kazhdan, Alexander Lubotzky}
\maketitle
\begin{abstract}
Expander graphs have been a focus of attention in computer science in the last four decades. In recent years a high dimensional theory of expanders is emerging. There are several possible generalizations of the theory of {\em expansion} to simplicial complexes, among them stand out {\em coboundary expansion} and {\em topological expanders}. It is known that for every $d$ there are {\em unbounded degree} simplicial complexes of dimension $d$ with these properties. However, a major open problem, formulated by Gromov, is whether {\em bounded degree} high dimensional expanders, according to these definitions, exist for $d \geq 2$. We present an explicit construction of bounded degree complexes of dimension $d=2$ which are high dimensional expanders. More precisely, our main result says that the $2$-skeletons of the $3$-dimensional Ramanujan complexes are topological expanders. Assuming a conjecture of Serre on the congruence subgroup property, infinitely many of them are also coboundary expanders.

%A family of $d$-dimensional simplicial complexes $\{X_{\alpha}\}$ is said to have the {\em topological overlapping property} if for very {\em continuous} map $f:X_{\alpha} \rightarrow \R^d$, there exists $z \in \R^d$ which is covered by constant fraction of the simplicies of $X_{\alpha}$.
%Expander graphs have this property for $d=1$. Gromov asked whether there exist {\em bounded degree} complexes of dimension $d \geq 2$ with the topological overlapping property.
%
%Independently, Linial and Meshulam studied high dimensional notion of expansion called {\em coboundary expansion} and asked whether bounded degree coboundary expanders exist for $d \geq 2$.
%
%We show the first bounded degree complexes of dimension $d=2$ which have the topological overlapping property. Assuming Serre's conjecture on the congruence subgroup property, our complexes are also bounded degree coboundary expanders. Our results are obtained by showing "property testing results": the cocycles in these complexes are testable within the space of cochains.

\end{abstract}


%\begin{abstract}
%Gromov and independently Linial and Meshulam suggested two related notions of high dimensional expanders. Gromov defined {\em topological expanders} while Linial and Meshulam defined {\em coboundary expanders}. Bounded degree high dimensional expanders of both types were not known to exist. We provide the first construction of {\em bounded degree} topological expanders. We show that assuming Serre's conjecture, our expanders are also bounded degree coboundary expanders. Our proof is inspired by the relation we have discovered between high dimensional expanders and property testing.
%%Several isoparametric inequalities for Ramanujan Complexes are proven, which gives some linear size bounds on the $\F_2$ systolic invariants of these simplicial complexes. We deduce that there exist $2$-dimensional simplicial complexes of bounded degree with the topological overlapping property. This answers a question asked by Gromov~\cite{Gromov}.
%\end{abstract}

\section{Introduction}\label{section:intro}
In the last four decades expander graphs have played an important role in computer science (see~\cite{HooryLinialWigderson} and the references therein) and more recently also in pure mathematics (see~\cite{LubotzkyAMS}). In the last few years a high-dimensional theory of expanders is starting to emerge (see~\cite{LubotzkyJapan} and the references therein). This theory has already found some applications in computer science (e.g. property testing~\cite{KaufmanLubotzkyPT}), combinatorics (e.g. random simplicial complexes~\cite{LinialMeshulam},~\cite{MeshulamWallach}), computational geometry (\cite{NewmanRabinovich}) and in topology (overlapping properties~\cite{Gromov}). Among the possible extensions of the notion of expanders to high dimensional simplicial complexes stand out the notions of {\em topological expanders} and {\em coboundary expansion} (whose definition is a bit technical and needs the language of cohomology, but it has already been proven to be useful in all these areas.).
%: (a) The study of random simplicial complexes, \'{a} la the model of Linial and Meshulam (see ~\cite{LinialMeshulam} and~\cite{MeshulamWallach}) (b) Property testing (see~\cite{KaufmanLubotzkyPT}) and the deepest aspect: (c) Topological overlapping of complexes~\cite{Gromov}.
%\tnote{elaborate more on the relation to property testing?}
%\tnote{What about Ilan Newman, that claims to have defined high-dim exp?}

Let us elaborate on coboundary expanders, on the topological overlapping and on the connection between them. The notations and terminology used below will be explained in details in Sections~\ref{section:coboundary-expansion} and~\ref{section:topological-overlapping}.

Recall that a {\em simplicial complex} $X$ on a set of vertices $V$ is a collection of finite subsets of $V$ (called {\em faces}) closed under inclusion, i.e., if $G \subseteq F \in X$ then $G \in X$. For $F \in X$, $\mbox{dim}(F) := |F|-1$ and $X$ is of {\em dimension} $d$, if $\mbox{max}\{\mbox{dim}(F) \mbox{ } | \mbox{ } F \in X\} = d$. It is a {\em pure} simplicial complex of dimension $d$ if all maximal faces (called {\em facets}) of $X$ are of cardinality $d+1$. Given such a simplicial complex one can associate with it a hypergraph $H=\tilde{X}$ with the set of vertices $V$ and the (hyper) edges of $H$ are the facets of $X$. So, $H$ is a $(d+1)$-uniform hypergraph. Given a $(d+1)$-uniform hypergraph $H$, one can associate with it a pure simplicial complex $X$ of dimension $d$: $X$ will be the collection of all subsets of the (hyper) edges of $H$. We see, therefore, that uniform hypergraphs and pure simplicial complexes are actually the same. It will be more convenient for us to work with simplicial complexes.

Given a finite simplicial complex of dimension $d$, we denote by $X(i)$ the set of its faces of dimension $i$, for $i=-1,0, 1,\cdots, d$. So, $X(-1) = \{\emptyset\}$, $X(0) = V$ - the vertices, $X(1)$ - the edges, $X(2)$ - the triangles etc. Given $\sigma \in X(i)$, let $c(\sigma) = \# \{\tau \in X(d) \mbox{ } | \mbox{ } \sigma \subseteq \tau\}$ and  $wt(\sigma) = \frac{c(\sigma)}{{d+1 \choose i+1}|X(d)|}$,  so $\sum_{\sigma \in X(i)}wt(\sigma)=1$.

Let $C^i=C^i(X,\F_2)$ be the $i$-cochains of $X$ over $\F_2$, i.e., the $\F_2$ vector space of all functions $f:X(i) \rightarrow \F_2$. Every such an $f$ can be considered also as a subset of $X(i)$. Let $\delta_i:C^i \rightarrow C^{i+1}$ be the $i$-coboundary map,
$$\delta_i(f)(\tau) = \sum_{\sigma \subseteq \tau, \sigma \in X(i)} f(\sigma), \mbox{ for } \tau \in X(i+1).$$

It is well known (and easy to prove) that $\delta_i \circ \delta_{i-1} = 0$, hence $B^i(X,\F_2) = \mbox{Image}(\delta_{i-1})$, the {\em $i$-coboundary} space, is contained in $Z^i(X,\F_2) = \mbox{Ker}(\delta_i)$, the {\em $i$-cocycle} space. The quotient space $H^i=Z^i/B^i$ is the {\em $i$-cohomology} group of $X$.

If $X$ is of dimension $1$, i.e., $X$ is a graph, one can easily check that $B^0=\{0,\mathbf{1}\}$ (where $\mathbf{1}$ is the constant function) and $H^0 = 0$ if and only if $X$ is connected. Moreover, $B^1$ is exactly the collections of "cuts" of $X$, namely, if $f \in C^0$ (so $f$ can be thought of as a subset $A$ of $V = X(0)$), $\delta(f)$ is $E(A, \bar{A})$, the set of edges from $A$ to its complement.

Finally, observe that the weight function on $X(i)$ induces a "norm" on $C^i(X,\F_2)$ defined by: $||f|| := \sum_{\sigma \in f}wt(\sigma)$. If $f \in C^i(X,\F_2)$, we denote $[f] := f+B^i(X,\F_2)$, i.e., the coset of $f$ modulo $B^i(X,\F_2)$ and the norm of the coset $||[f]||: = \mbox{min}\{||g|| \mbox{ } | \mbox{ } g \in [f]\}$. One can easily checks that $||[f]||$ is actually equal to "the distance" between $f$ and $B^i(X,\F_2)$ in terms of the norm $|| \cdot ||$.

\begin{definition}~\label{def-coboundary-exp} Let $X$ be a finite $d$-dimensional simplicial complex and $0 \leq i \leq d-1$. The $i$-th {\em $\F_2$-coboundary expansion} $\epsilon_i$ of $X$ is defined as
$$\epsilon_i = \mbox{min} \{\frac{||\delta_i f ||}{||[f]||} | f \in C^i \setminus B^i\},  $$

A family $\{X_a\}_{a\in A}$ of $d$-dimensional pure complexes is called $\epsilon$-{\em coboundary expander(s)} if there exists $\epsilon > 0$ such that $\epsilon_i(X_a) \geq \epsilon$ for every $0 \leq i \leq d-1$ and every $a \in A$.
\end{definition}

The reader is encouraged to check the case in which $X$ is a $1$-dimensional simplicial complex, i.e., a graph. In this case $\epsilon_0$ is equal to the normalized Cheeger constant of the graph.

%A family of pure simplicial complexes of dimension $d$ will be called {\em coboundary expanders} if they are all $\epsilon$-coboundary expanders for some $\epsilon > 0$.

The notion of coboundary expanders is (essentially) due to Linial-Meshulam~\cite{LinialMeshulam} and Gromov~\cite{Gromov} (see also~\cite{MeshulamWallach} and~\cite{DK}). But, there are other ways to generalize expander graphs to hypergraphs.

We pass now to "overlapping properties" and to geometric and topological expanders.

\begin{definition} Let $X$ be a finite $d$-dimensional pure simplicial complex.
\begin{enumerate}
\item We say that $X$ has the {\em $\epsilon$-geometric overlapping property} if for every map $f:X(0) \rightarrow \R^d$, there exists a point $z \in \R^d$ which is covered by at least $\epsilon$-fraction of the images of the faces in $X(d)$ under $\tilde{f}$, where $\tilde{f}$ is the unique affine extension of $f$ to a map from $X$ to $\R^d$.
\item We say that $X$ has the {\em $\epsilon$-topological overlapping property} if the same conclusion holds for every $f$ and every $\tilde{f}$, where this time $\tilde{f}$ is any continuous extension of $f$ to a map from $X$ to $\R^d$.
\end{enumerate}
\end{definition}

A family of $d$-dimensional pure simplicial complexes is a family of {\em geometric (resp. topological) expanders} if they have the $\epsilon$-geometric (resp. topological) overlapping property for the same $\epsilon > 0$.

Clearly topological expanders are geometric expanders.

Expander graphs have the topological overlapping property. Indeed, if $f:V \rightarrow \R$ maps the vertices of $V$ into the real line, then a point $z \in \R$ which is chosen so that half of $f(V)$ is above $z$ and half below, is covered by a constant fraction of the images of the edges. Thus, the overlapping property can also be considered as an extension of expansion from graphs to simplicial complexes.

A classical result of Boros and F\"{u}redi~\cite{BorosFuredi} (for $d=2$) and B\'{a}r\'{a}ny~\cite{Barany} (for general $d \geq 2$) asserts that there exists $\epsilon_d >0$ such that given any set $P$ of $n$ points in $\R^d$, there exists $z \in \R^d$ which is contained in at least $\epsilon_d$-fraction of the ${n \choose d+1}$ simplicies determined by $P$. So, Barany's theorem is the statement that $\Delta_n^{(d)}$ - the complete $d$ - dimensional simplicial complex on $n$ vertices of dimension $d$-are geometric expanders. Gromov proved the remarkable result, that they also have the topological overlapping property! (The reader is encouraged to think about the case $d=2$ to see how non-trivial is this result and even somewhat counter intuitive!)
%Moreover, he went ahead and showed that various other families of simplicial complexes (of fixed dimension $d$) have the topological overlapping property, e.g., spherical buildings (see also~\cite{LubotzkyMeshulamMozes}). Here we say that a family of simplicial complexes of dimension $d$ has the geometric (resp. topological) overlapping property if they all have it with the same $\epsilon > 0$.

There are several methods to show "geometric overlapping". On the other hand, all the results on "topological overlapping" are derived via the following theorem of Gromov, which makes a connection between coboundary expansion and topological expanders.

\begin{thm}[coboundary expanders are topological expanders](\cite{Gromov}, see~\cite{KaufmanWagner} for a simplified proof)\label{thm:Gromov}
If $X$ is a finite simplicial complex of dimension $d$, with $\epsilon_i(X) \geq \epsilon > 0$ for every $i=0,\cdots,d-1$, then $X$ has $\epsilon'$-topological overlapping property for some $\epsilon' > 0$ depending on $d$ and $\epsilon$.
\end{thm}

So, coboundary expanders are topological expanders (and hence also geometric expanders). The complete $d$-dimensional complexes on $n$ vertices ($d$ fixed, $n \rightarrow \infty$) are coboundary expanders (this was proved in~\cite{MeshulamWallach} and~\cite{Gromov} independently). Similarly, the finite spherical buildings are coboundary expanders (\cite{Gromov}, see~\cite{LubotzkyMeshulamMozes} for a proof and a generalization to base-transitive matroids). In~\cite{LubotzkyMeshulam} a random model of $2$-dimensional simplicial complexes is given, with a complete $1$-skeleton, based on latin squares. This gives coboundary expanders of bounded {\em edges} degree. But all the known examples so far have unbounded vertex degree.

This raise the very basic question of the existence of {\bf high dimensional bounded degree expanders}. It is interesting to compare this with the one dimensional case, i.e., graphs.

It is trivial to show that the complete graphs are expanders. Barany's theorem is exactly the statement that the complete complexes are geometric expanders, and Gromov's deep result gives that they are also topological expanders. It is less trivial that there are families of expander graphs of bounded degree, but by now, there is a good number of methods to show that: random methods, property ($T$), Ramanujan conjecture, the zig-zag product and interlacing polynomials (see~\cite{LubotzkyBook},~\cite{HooryLinialWigderson}~\cite{MarcusSpielmanSrivastavaInterlacing} and the references therein).

In the higher dimensional case the situation is much more difficult and very little is known in this direction. The only case in which there are satisfactory answers is the case of geometric expanders: In~\cite{FGLNP} it is shown by random and explicit methods that geometric expanders of bounded degree do exist. But the examples there are not coboundary expanders and it is not known if they are topological expanders. The following two basic problems were left open:

\begin{question}\label{question:bounded-deg-top-overlapping}(Gromov~\cite{Gromov})  For a fixed $d \geq 2$, is there an infinite family of $d$-dimensional {\bf bounded degree topological expanders}?
\end{question}

\begin{question}\label{question:bounded-deg-coboundary-exp}(see~\cite{DK}) For a fixed $d \geq 2$, is there an infinite family of $d$-dimensional {\bf bounded degree $\F_2$-coboundary expanders}?
\end{question}

Here by bounded degree we always mean that every vertex is contained in a bounded number of faces.

Of course, Theorem~\ref{thm:Gromov} shows that a positive answer to Question~\ref{question:bounded-deg-coboundary-exp} implies also an affirmative answer to Question~\ref{question:bounded-deg-top-overlapping}. The questions have been completely open even for random methods. %Even random constructions are not known to answer Questions~\ref{question:bounded-deg-top-overlapping} and~\ref{question:bounded-deg-coboundary-exp}.

%\tnote{even random constructions are not known; our construction is explicit}.

The main goal of this paper is to announce and to sketch a proof for a positive answer to Question~\ref{question:bounded-deg-top-overlapping} for $d=2$. This gives the first examples of high dimensional bounded degree topological expanders. Along the way we will also see that if one accepts a special case of Serre's conjecture on the congruence subgroup property, then we get also an affirmative answer to Question~\ref{question:bounded-deg-coboundary-exp} for $d=2$.  More details will be given in Section~\ref{section:isoparametric} and Section~\ref{section:proof-main-thm}. Let us now give only the main points.

In~\cite{Gromov}, Gromov suggested that the Ramanujan complexes (see~\cite{LSV1}) of dimension $2$ have the topological overlapping property, and proved a partial result in this direction. (To be more precise, he proved this property when $\tilde{f}$ in Definition~\ref{def-coboundary-exp} is assumed to be at most $k$ to $1$, for bounded $k$). We fell short from proving this, but we prove:

\begin{thm}\label{thm:main-thm}
Let $q$ be a sufficiently large prime power, and let $F=\F_q((t))$, the field of Laurent power series over the finite field $\F_q$. Let $\{Y_{a}\}_{a \in A}$ be the family of $3$-dimensional non-partite Ramanujan complexes obtained from the Bruhat-Tits building associated with $PGL_4(F)$ (see~\cite{LSV1,LSV2}). For each such $Y_{a}$, let $X_{a} = Y_{a}^{(2)}$ - the $2$-skeleton of $Y_{a}$. Then, the family of $2$-dimensional simplicial complexes $\{X_{a}\}_{a \in A}$ is an infinite family of topological expanders of degree $O(q^5)$ (i.e., every vertex is contained in at most $O(q^5)$ simplicies).
\end{thm}

In spite of the fairly abstract formulation of Theorem~\ref{thm:main-thm}, let us stress that it gives an explicit construction of examples of topological expanders. A detailed description of the Ramanujan complexes $Y_{a}$'s of the theorem, as Cayley complexes of dimension $3$ of specific finite groups ($PSL_4(q^e)$ in our case) with explicit sets of generators, is given in~\cite[Section 9]{LSV2}. Recall that a $d$-dimensional Cayley complex of a group $G$ with respect to a symmetric set of generators $S$, is the simplicial complex whose set of vertices is $G$ and for $i \leq d$, $\{g_0,\cdots, g_i\}$ forms an $i$-face if for every $0\leq t\neq s \leq i$, $g_t^{-1}g_s \in S$. This is a clique complex of the Cayley graph $Cay(G:S)$, of dimension $d$.
The $X_{a}$'s are simply the $2$-skeletons of $Y_{a}$'s, i.e., ignoring the $3$ simplices. So, the $X_{a}$'s are Cayley complexes of dimension $2$.
Presenting the details of the exact construction requires a lot of notation, so we refer the reader to~\cite{LSV2}.

%We refer the reader to section~\ref{section:isoparametric},~\ref{section:proof-main-thm} and~\ref{section:isoparametric-proof} where we will say more %on the notions mentioned in the theorem as well as sketch the proof.

In Section~\ref{section:isoparametric} we will elaborate on Ramanujan complexes and in Sections~\ref{section:proof-main-thm} and~\ref{section:isoparametric-proof} we sketch the proof. Crucial ingredients in the proofs are the facts that the $1$-skeleton of $Y_{a}$ is nearly a Ramanujan graph and the links of $Y_{a}$ are coboundary expanders.

We only mention here that the main technical tools are new $\F_2$ isoperimetric inequalities.
Such inequalities are relevant to classical and quantum error correcting codes (compare with~\cite{Zemor} and~\cite{LubotzkyGuth}).
In fact, our first motivation to study these inequalities came from coding theory. We hope to come back to this direction in future works.

Here we should stress that in general the $X_{a}$'s in the theorem are {\bf not} $\F_2$-coboundary expanders. This is due to the fact (see~\cite{KaufmanKazhdanLubotzky} for a proof) that it is possible that $H^1(X_{a},\F_2) \neq 0$, while for coboundary expanders $X$, the $i$-cohomology group over $\F_2$ of $X$ must vanish for $0 \leq i < \mbox{dim} X$. On the other hand, our proof will show that this is the only obstacle for our complexes $X_{a}$ to be coboundary expanders.

Now, if the arithmetic lattices used in~\cite{LSV2} to construct the Rammanujan complexes $Y_{a}$ in Theorem~\ref{thm:main-thm} (the so called "Cartwright-Steger lattices") satisfy the congruence subgroup property, then one can deduce that for infinitely many of the $X_{a}$'s, $H^1(X_{a},\F_2)$ does vanish. These $X_{a}$'s are therefore also coboundary expanders of dimension $2$.
According to a well known general conjecture of Serre~\cite[p.489]{Serre}, the Cartwright-Steger arithmetic groups, being lattices in high rank Lie groups, should indeed satisfy the congruence subgroup property. The general conjecture of Serre has been proven in most cases, but unfortunately, not yet for the Cartwright-Steger lattices. So, in summary, assuming Serre's conjecture (in fact, a very special case of it and furthermore, a weak form of it for this special case-see~\cite{KaufmanKazhdanLubotzky}) our work gives a positive answer also to Question~\ref{question:bounded-deg-coboundary-exp} for $d=2$. Unconditionally, we give a positive answer to Question~\ref{question:bounded-deg-top-overlapping} for $d=2$. More details are given in Section $5$ and a full proof in~\cite{KaufmanKazhdanLubotzky}.

\section{$\F_2$-coboundary expansion}\label{section:coboundary-expansion}

%Let $X$ be a finite simplicial complex.
%so $X$ is a non-empty collection of subsets of the "set of vertices" $V=X(0)$, satisfying $F \in X$ and $G \subseteq F$ implies $G \in X$. Each $F \in X$ is called a {\em face} or {\em simplex} of dimension $\mbox{dim}(F) = |F|-1$. By definition, $d=\mbox{dim}X$ is the maximum dimension of its faces. We assume throughout that $X$ is pure, i.e., all maximal faces are of dimension $d$. We denote, for $-1 \leq i \leq d$, by $X(i)$, the set of $i$-faces, i.e., those of dimension $i$, and $X^{(i)}$ the $i$-skeleton. Of course,
%One can think of $X$ also as a topological space by associating with every face an affine simplex and gluing them together in the obvious way.
To any finite simplicial complex one associates a geometric realization, which is obtained by gluing the affine simplicies along common faces.
So one can talk about affine and continuous functions on $X$.

%By $C^i=C^i(X,\F_2)$ we denote the $i$-cochains over the field $\F_2$, i.e., the functions $\alpha$ from $X(i)$ to $\F_2$. Every such function can also be considered as a subset of $X(i)$.

The {\em $i$-skeleton} of $X$ is the sub-complex $X^{(i)}:= \cup_{-1 \leq j \leq i}X(j)$. Given $\tau \in X(i)$, the {\em link of $X$ at $\tau$} denoted $X_{\tau}$ is the collection of all sets of the form $\sigma \setminus \tau$, where $\sigma \in X$ and $\tau \subseteq \sigma$. Thus, $X_{\tau}$ is a complex of dimension $\mbox{dim}(X) - \mbox{dim}(\tau) -1 = d-i-1$. In particular, for a vertex $v$, the link $X_v$ of $v$ is of dimension $d-1$. A cochain $\alpha \in C^i(X,\F_2)$ defines a cochain $\alpha_v \in C^{i-1}(X_v,\F_2)$ by $\alpha_v(\sigma \setminus \{v\}) := \alpha(\sigma)$ for $\sigma \in X(i)$ containing $v$.

For $\sigma \in X(i)$, denote $c(\sigma):=|\{\tau \in X(d) \mbox{ } | \mbox{ }  \sigma \subseteq \tau\}|$ and $wt(\sigma):=\frac{c(\sigma)}{{d+1 \choose i+1} |X(d)|}$, so $\sum_{\sigma \in X(i)}wt(\sigma) = 1$. For $\alpha \in C^i(X,\F_2)$ we define $||\alpha||:= \sum_{\sigma \in X(i), \alpha(\sigma) \neq 0}wt(\sigma)$. One easily checks that $||.||$ is a "norm" such that $||\alpha|| =0$ iff $\alpha=0$ and $||\alpha_1+ \alpha_2|| \leq ||\alpha_1|| + ||\alpha_2||$ for $\alpha_1, \alpha_2 \in C^i(X,\F_2)$.

A cochain $\alpha \in C^i(X,\F_2)$ is called {\em minimal} if it is of minimal norm within its class $[\alpha]$ modulo $B^i(X,\F_2)$, i.e., $||\alpha|| = ||[\alpha]||$. It is called {\em locally minimal} if for every $v \in X(0)$, $\alpha_v$ is a minimal cochain in $C^{i-1}(X_v, \F_2)$. A minimal cochain is always locally minimal, but not vice versa.

%Throughout this paper we assume that the space $C^i$ has a norm $||\alpha|| \in \R^{ \geq 0}$, satisfying $||\alpha_1 + \alpha_2|| \leq ||\alpha_1 || + ||\alpha_2||$. We also assume that there exists some $0 < M \in \R$ such that for every $0 \leq i \leq d$ and $\alpha \in C^i(X, \F_2)$, $\frac{1}{M}||\alpha||_0 \leq ||\alpha|| \leq M||\alpha||_0$ where $||\alpha||_0 = \frac{|\alpha|}{|X(i)|}$ is the normalized counting norm. In this case we say that $||\cdot||$ is {\em $M$-bounded}. Given a norm on $X$ and $v$ a vertex of $X$, we get an induced norm $||\cdot||_v$ on $X_v$, where for $\gamma \in C^{i-1}(X_v, \F_2)$, $||\gamma||_v = \frac{|X(i)|}{|X_v(i-1)|} \cdot ||\tilde{\gamma}||$ where $\tilde{\gamma}$ is the $i$-cochain of $X$ which is equal to $\gamma$ on the link at $v$ and zero otherwise.

The coboundary map $\delta =\delta_i:C^i \rightarrow C^{i+1}$ is defined as
$$\delta(\alpha)(F) := \sum_{G \subseteq F, |G|=|F|-1}\alpha(G)$$
when $\alpha \in C^i$ and $F \in C^{i+1}$. It is easy to check that $\delta_{i+1} \circ \delta_i =0$.
Hence, $B^i=\mbox{Im}(\delta_{i-1}) \subset Z^i =\mbox{Ker}(\delta_{i})$. The quotient group $Z^i/ B^i = H^i(X,\F_2)$ is the $i$-th cohomology group of $X$. The elements of $B^i$ (resp. $Z^i$) are called $i$-coboundaries (resp. $i$-cocycles).

We can now define the $i$-expansion $\epsilon_i(X)$ of $X$ with respect to the norm $||.||$.

\begin{definition} .

\begin{enumerate}
\item($\F_2$-coboundary expansion) For $i=0,1,\cdots, d-1$, denote
$$\epsilon_i(X) := \mbox{min}\{\frac{||\delta_i \alpha  ||}{||[\alpha] ||} \mbox{ } |  \mbox{ }\alpha \in C^i \setminus B^i \}$$
When $[\alpha] = \alpha+B^i$ and $||[\alpha]|| = \mbox{min}\{||\gamma|| \mbox{ } |  \mbox{ } \gamma \in [\alpha]\}$.

\item($\F_2$-cocycle expansion) For $i=0,1,\cdots, d-1$, denote
$$\tilde{\epsilon}_i(X) := \mbox{min}\{\frac{||\delta_i \alpha  ||}{||\{\alpha\} ||} \mbox{ } |  \mbox{ }\alpha \in C^i \setminus Z^i \}$$
When $\{\alpha\} = \alpha+Z^i$ and $||\{\alpha\}|| = \mbox{min}\{||\gamma|| \mbox{ } |  \mbox{ } \gamma \in \{\alpha\}\}$.

\item(cofilling constant) The $i$-th cofilling constant of $X$, $0 \leq i \leq d$ is
$$\mu_i(X) := \mbox{max}_{0 \neq \beta \in B^{i+1}}\{\frac{1}{||\beta||} \mbox{min}_{\alpha \in C^i, \delta \alpha = \beta} ||\alpha||\}$$.
\end{enumerate}
\end{definition}

If $\{X_j\}_{j \in J}$ is a family of $d$-dimensional simplicial complexes with $\epsilon_i(X_j) \geq \epsilon$ (resp. $\tilde{\epsilon_i}(X_j) \geq \epsilon$) for some $\epsilon > 0$ and every $0 \leq i \leq d-1$ and every $j\in J$, we say that this is a family of {\em coboundary expanders} (resp. {\em cocycle expanders}). Note that $\{X_j\}_{j \in J}$ is a family of cocycle expanders iff there exists $M \in \R$ such that $\mu_i(X_j) \leq M$ for every $i=0, \cdots, d-1$ and $j \in J$.

We believe (see below) that Ramanujan complexes are cocycle expanders but it is shown in~\cite{KaufmanKazhdanLubotzky} that in general they are {\em not} coboundary expanders.

Note also that for $i=-1$, $X(-1)=\{\emptyset\}$ and hence $B^0=\{\mathbf{0},\mathbf{1}\}$, where $\mathbf{1}$ is the constant function. On the other hand $B^1$ is the set of all "cuts" in the $1$-skeleton of $X$. Indeed, if $\alpha \in C^0$, then $\alpha$ is a characteristic function of some subset $A$ of $V = X(0)$ and one easily see that $\delta(\alpha) \in C^1$ is exactly the set of edges from $A$ to $\bar{A}$. Moreover, the coset of $\alpha$ modulo $B^0$, namely $\alpha+B^0$ consists of two elements: $\alpha$ and $\alpha + \mathbf{1}$ or in terms of subsets, $A$ and $\bar{A}$. Recall that the Cheeger constant of a graph is, by definition,
$$h(X):= \mbox{min}_{\emptyset \neq A \subset V, |A| \leq \frac{|V|}{2}}\frac{|E(A,\bar{A})|}{|A|}.$$

One can now easily check:

\begin{proposition}
\label{prop:cofilling-vs-coboundary-exp}
The following hold:
\begin{enumerate}
\item For $0 \leq i \leq d-1$, $\epsilon_i > 0$ iff $H^i(X,\F_2)=0$. In particular, $H^i=0$ for coboundary expanders.
\item Always $\mu_i=\frac{1}{\tilde{\epsilon}_i}$. So if $H^i(X,\F_2)=0$ then $\epsilon_i = \tilde{\epsilon}_i$ and $\mu_i = \frac{1}{\epsilon_i}$.
\item For a regular graph $X$, $ \epsilon_0(X) = \frac{|X(0)|}{|X(1)|} \mbox{min}_{\emptyset \neq A \subset V, |A| \leq \frac{|V|}{2}}\frac{|E(A,\bar{A})|}{|A|} = h(X)\frac{|X(0)|}{|X(1)|}$,
where $E(A,\bar{A})$  is the set of edges from $A$ to $\bar{A}$. So, $\epsilon_0(X)$ is the normalized Cheeger constant.
\end{enumerate}
\end{proposition}

Proposition~\ref{prop:cofilling-vs-coboundary-exp} explains why we can consider the $\epsilon_i$'s as expansion constants of $X$. In a way $\tilde{\epsilon}_i$ capture a situation which for graphs means that the graph is not necessarily connected, but each connected component is an expander.

\section{Topological overlapping}\label{section:topological-overlapping}
In Section~\ref{section:intro} we saw Gromov's Theorem~\ref{thm:Gromov} saying that coboundary expanders are topological expanders. For Ramanujan complexes, it is possible that $H^i \neq 0$, and so they are in general {\em not} coboundary expanders. It was noted by Kaufman and Wagner~\cite{KaufmanWagner} that one may give a more general version of Gromov's Theorem that will also work when $H^i \neq 0$. If $H^i \neq 0$, then one should assume a linear systolic inequality (condition (\ref{item:thm:Gromov-systolic:second}) below) for the non-trivial $i$-cocycles of $Z^i(X,\F_2)$. Here is the theorem in its stronger form.

\begin{thm}\label{thm:Gromov-systolic}
Let $X$ be a $d$-dimensional pure simplicial complex of dimension $d$ and $0 < \mu, \eta \in \R$. Assume
\begin{enumerate}
%\item\label{item:thm:Gromov-systolic:zero} For every $0 \leq i \leq d$, the norm on $C^i(X,\F_2)$ is $M$-bounded.
\item\label{item:thm:Gromov-systolic:first} For every $0 \leq i \leq d-1$, $\mu_i(X) \leq \mu$.
\item\label{item:thm:Gromov-systolic:second} For every $0 \leq i \leq d-1$ and every $\alpha \in Z^i(X,\F_2) \setminus B^i(X,\F_2)$, $||\alpha|| \geq \eta$.
\end{enumerate}
Then there exists $c=c(d,\mu,\eta)$ so that $X$ has $c$-topological overlapping.
\end{thm}

In other words,  a family of $d$-dimensional expanders which are cocycle expanders and satidfay "linear systolic inequality" forms a family of topological expanders. The reader may note that the systolic condition (\ref{item:thm:Gromov-systolic:second})  for graphs means that even if the graph is not connected, every connected component is large (and, in particular, there are only bounded number of connected components). Indeed, this plus condition (\ref{item:thm:Gromov-systolic:first}), which as said before, ensures that every connected component is an expander, suffice to deduce topological overlapping for graphs.

For a proof of Theorem~\ref{thm:Gromov-systolic} see~\cite{KaufmanWagner}. Note that if $H^i = 0$ for every $i=0, \cdots, d-1$, condition (\ref{item:thm:Gromov-systolic:second}) above is vacuous and Theorem~\ref{thm:Gromov-systolic} is the same as Theorem~\ref{thm:Gromov}.

\section{Isoperimetric inequalities for Ramanujan complexes}\label{section:isoparametric}
A finite connected $k$-regular graph $Y$ is called a {\em Ramanujan graph} if every eigenvalue $\lambda$ of the adjacency matrix $A=A_Y$ of $Y$ satisfies either $|\lambda|=k$ or $|\lambda| \leq 2\sqrt{k-1}$. By Alon-Boppana Theorem, these are the optimal expanders (at least from a spectral point of view). See for example~\cite{LubotzkyBook} and the references therein, for explicit construction of such graphs as quotients of the Bruhat-Tits tree $\calB_2(F)$ associated with the group $PGL_2(F)$ when $F$ is a local field (e.g. $\Q_p$-the $p$-adic numbers or $\F_q((t))$). Ramanujan graphs are obtained by taking the quotients of $\calB_2(F)$ modulo the action of suitable congruence subgroups $\Gamma$ of an arithmetic cocompact discrete subgroup ($=$lattice) $\Gamma_0$ of $PGL_2(F)$.

The above theory and constructions have been generalized to the higher dimensional case. In~\cite{LSV1} the notion of Ramanujan complexes was defined and explicit constructions of such $d$-dimensional complexes were given in~\cite{LSV2}. This time the complex $Y$ is obtained as a quotient of the Bruhat-Tits building associated with $PGL_{d+1}(F)$, modulo the action of a suitable congruence subgroup $\Gamma$ of an arithmetic lattice $\Gamma_0$ in $PGL_{d+1}(F)$. In~\cite{LSV2} a specific arithmetic lattice $\Gamma_0$ was used. This is the remarkable lattice constructed by Cartwright and Steger~\cite{CS} which acts simply transitive on the vertices of the building.

To keep the exposition simple we will work only with those congruence subgroups of $\Gamma_0$ which destroy completely the coloring of the building (this is an analogue of the non-bipartite case of Ramanujan graphs - see a discussion in~\cite{EvraGLubotzky}). When $\Gamma_0$ is the Cartwright-Steger arithmetic lattice, there are infinitely many congruence subgroups of $\Gamma_0$ with this property~\cite{LSV2}, so the associated complex $\calB \backslash \Gamma$ is non-partite (see~\cite{EvraGLubotzky}).

The reader is referred to these papers (and to a more reader friendly description in~\cite{LubotzkyJapan}) for more details and references. Rather than repeating the definition of Ramanujan complexes, let us give here the few properties of such complexes that we will be using.

Assume $Y$ is a finite non-partite Ramanujan complex of dimension $d$ obtained as a quotient of $\calB_{d+1}(F)$ where $F$ is a local field with residue field $\F_q$. The relevant properties are the following:
\begin{description}
\item[($A$)] The $1$-skeleton $Y^{(1)}$ of $Y$ is a $k$-regular graph when $k=\sum_{i=1}^{d}{d+1 \choose i}_q$, where ${d+1 \choose i}_q$ is the number of subspaces of $\F_q^{d+1}$ of dimension $i$. The non trivial eigenvalues $\lambda$ of $A_{Y^{(1)}}$ satisfy $|\lambda| \leq \sum_{i=1}^{d} {d+1 \choose i}q^{\frac{i(d+1-i)}{2}}$. Thus, $|\lambda| \leq c(d) k^{\frac{1}{2}}$ when $c(d)$ depends only on $d$.
\item[($B$)] The link $Y_v$ of every vertex $v$ is isomorphic to the flag complex of $F_q^{d+1}$. This is the complex of dimension $d-1$ whose vertices are all the non-trivial proper subspaces of $\F_q^{d+1}$ and $\{w_0, \cdots, w_i\}$ forms an $i$-cell if, possibly after reordering,  $w_0 \subset w_1 \subset \cdots \subset w_i$.
\end{description}

Let us spell out the properties for the case we are most interested in this paper - the $3$ dimensional Ramanujan complexes. So here $d=3$ and we have
%\begin{itemize}
\begin{description}
\item[($A_3$)] $Y^{(1)}$ is the $k$-regular graph with $k = {4 \choose 1}_q +{4 \choose 2}_q +{4 \choose 3}_q =2 \frac{q^4-1}{q-1} + \frac{q^4-1}{q-1} \cdot \frac{q^3-1}{q^2-1} \approx q^4$, and every eigenvalue $\lambda$ of $A_{Y^{(1)}}$ is either $k$ or $|\lambda| \leq cq^2$.
\item[($B_3$)] The link $Y_v$ of a vertex is isomorphic to the flag complex of $\F_q^4$. It has therefore vertices of $3$ types $M_1 \cup M_2 \cup M_3$ (corresponding to subspaces of $\F_q^4$ of dimension $1,2,3$). The degree of the vertices in $M_2$ is $2(q+1)$, while those in $M_1 \cup M_3$ have degree $2(q^2+q+1)$. This means that $Y$ has edges of two types: "black" - the ones correspond to $M_1 \cup M_3$, such an edge lies on $2(q^2+q+1)$ triangles and "white" - those correspond to $M_2$ and each of them lies on $2(q+1)$ triangles. Thus, if $e_b$ and $e_w$ are black and white edges, respectively, then $wt(e_b) = \Theta \cdot wt(e_w)$ where $\Theta =\frac{q^2+q+1}{q+1}$. Each triangle of $Y$ has two black edges and one white and it lies in $q+1$ pyramids (since each edge of $Y_v$ is in $q+1$ triangles). Thus, every vertex of $Y$ lie on $O(q^5)$ triangles.
\end{description}
%\end{itemize}

We can now state the main technical result of this paper.

\begin{thm}\label{thm:isoparametric}
Fix $q \gg 0$ (i.e. $q>q_0=q_0(3)$). Let $F=\F_q((t))$, $\calB = \calB_4(F)$ the $3$-dimensional Bruhat-Tits building associated with $PGL_4(F)$, and $Y=\Gamma \backslash \calB $ a non-partite Ramanujan complex. There exist $\eta_0,\eta_1,\eta_2, \epsilon_0,\epsilon_1,\epsilon_2$ all greater than $0$ such that: If $\alpha \in C^i(Y,\F_2)$ is a locally minimal cochain with $||\alpha|| \leq \eta_i$ then $||\delta_i(\alpha)|| \geq \epsilon_i||\alpha||$. These constants may depend of $q$ but not on $Y$.
\end{thm}

For $\eta_0$ we could take $\eta_0 = 1$, i.e., this is true for every $\alpha \in C^0$ (as anyway $||\alpha|| \leq  1$, and if $\alpha$ is locally minimal then even $||\alpha|| \leq \frac{1}{2}$). This is basically saying that the 1-skeleton is an expander. This is not so for $\eta_1$ and $\eta_2$.
In~\cite{KaufmanKazhdanLubotzky}, it is shown that it is possible that for $Y$ in Theorem~\ref{thm:isoparametric}, $H^1(Y,\F_2)$ and $H^2(Y,\F_2)$ be non zero. Hence, Theorem~\ref{thm:isoparametric} is not true without some assumptions such as $||\alpha|| \leq \eta_i$. In particular, $Y$ and even $X=Y^{(2)}$, are in general not coboundary expanders. We are still able to show that $X$ is a cocycle expander and satisfies a linear systolic inequality. Hence, it has the topological overlapping property (due to Theorem~\ref{thm:Gromov-systolic}); see Section~\ref{section:proof-main-thm} below. It should be mentioned, however, that this is the only obstacle, i.e., if $H^1(X,\F_2) (= H^1(Y,\F_2)) =0$ then $X$ is also a coboundary expander.

Now, by a general conjecture of Serre~\cite{Serre} (which has been proven in many cases!) the Cartwright-Steger lattice satisfies the congruence subgroup property. Unfortunately, the conjecture is still open for the Cartwright-Steger lattice, or for all the other arithmetic lattices coming from division algebras. Anyway, if one would assume Serre's conjecture for the Cartwright-Steger lattice, then there are infinitely many $Y$'s as above with $H^1(Y,\F_2) =0$ (see~\cite{KaufmanKazhdanLubotzky} for a proof).

Hence,

\begin{cor}\label{cor:assuming-serre's-conj}
Assuming Serre's conjecture on the congruence subgroup property for the Cartwright-Steger lattices, then for $q \gg 0$ (i.e. $q>q_0=q_0(3)$) , there are infinitely many Ramanujan complexes $Y$, quotients of $\calB_4(F)$, such that $X=Y^{(2)}$ are coboundary expanders.
\end{cor}
%\tnote{say more on what is the Serre's conj?}

We postpond the sketch of the proof of Theorem~\ref{thm:isoparametric} to Section~\ref{section:isoparametric-proof} and show first how the main theorem, Theorem~\ref{thm:main-thm}, can be deduced from Theorem~\ref{thm:isoparametric}.

\section{Proof of the main theorem}\label{section:proof-main-thm}
Recall the notations of the main theorem (Theorem~\ref{thm:main-thm}). We have to show that $X=X_a$ the $2$-skeleton of $Y=Y_a$, the Ramanujan complex of dimension $3$, has the topological overlapping property with a constant that may depend on $q$ but not on $a$. As explained in Section~\ref{section:isoparametric}, every vertex of $Y$ lies on $O(q^4)$ edges and on $O(q^5)$ triangles. So for a fixed $q$, we will get a family of bounded degree topological expanders.

To achieve the goal, we will prove that $X$ satisfies conditions (\ref{item:thm:Gromov-systolic:first}) and (\ref{item:thm:Gromov-systolic:second}) of Theorem~\ref{thm:Gromov-systolic} with constants $\mu$ and $\eta$ independent of $a$ (but may depend on $q$).

\begin{lemma} If $\alpha \in C^i(Y,\F_2)$ then there exists $\tilde{\alpha} \in \alpha + B^i(Y,\F_2)$ such that $\tilde{\alpha}$ is locally minimal, $||\tilde{\alpha}|| \leq ||\alpha||$ and $\alpha -\tilde{\alpha} =\delta_{i-1}(\gamma)$ for some $\gamma \in C^{i-1}(Y,\F_2)$ with $||\gamma|| \leq c||\alpha||$, where the constant $c$ may depend on $q$ but not on $Y$.
\end{lemma}

\begin{proof} If $\alpha$ is locally minimal, there is nothing to prove. If it is not, then at some vertex $v$, $\alpha_v$ is not minimal, i.e., there exists $\gamma_v \in C^{i-2}(Y_v,\F_2)$ such that $||\alpha_v +\delta_{i-2}(\gamma_v)|| < ||\alpha_v||$.
Let $\tilde{\gamma}$ be the $(i-1)$-cochain of $Y$ with $\tilde{\gamma}$ equal to $\gamma_v$ on all the $(i-1)$-cells containing $v$ (i.e., $\tilde{\gamma_v} = \gamma_v$) and zero elsewhere, so this $\tilde{\gamma}$ has at most $c_1(q)$ $(i-1)$-cochains. Then $\alpha' = \alpha + \delta_{i-1}(\tilde{\gamma})$ satisfies $||\alpha'|| < ||\alpha||$. If $\alpha'$ is locally minimal we stop and take $\tilde{\alpha} = \alpha'$. If not, we continue the process. Let us note that when we "correct" $\alpha$ at a vertex $v$, we may destroy it in a neighboring vertex $w$. Still, the process terminates since each time the new $\alpha$ has a strictly smaller norm than the norm of the previous $\alpha$. The number of possible values of the norm in the process is at most ${d+1 \choose i+1}|X(d)|||\alpha||$.  Now, each "move" change $\alpha$ by $\delta_{i-1}(\tilde{\gamma})$ when $\tilde{\gamma}$ has a bounded support. So altogether, $\tilde{\alpha} - \alpha =\delta_{i-1}(\gamma)$ when $|\gamma| \leq c_2(q)|\alpha|$. Thus, $||\gamma|| \leq c_3 ||\alpha||$, for some constant depending only on $q$.
\end{proof}

We need to show that $X$ of Theorem~\ref{thm:main-thm} satisfies Conditions (\ref{item:thm:Gromov-systolic:first}) and (\ref{item:thm:Gromov-systolic:second}) of Theorem~\ref{thm:Gromov-systolic}. Let $\alpha \in C^i(X,\F_2) = C^i(Y,\F_2)$. Assume $||\alpha|| < \eta_i$, where $\eta_i$ is from Theorem~\ref{thm:isoparametric}, then $||\tilde{\alpha}|| < \eta_i$, and by this theorem, $||\delta_i(\tilde{\alpha})|| \geq \epsilon_i||\tilde{\alpha}||$. Now, if $\alpha$ was in $Z^i(Y,\F_2)$ to start with, then so is $\tilde{\alpha}$ , hence $\delta_i(\tilde{\alpha}) =0$. This is a contradiction unless $\tilde{\alpha} = 0$, so, $\tilde{\alpha} \in B^i$, and therefore also $\alpha \in B^i$ , i.e., $\alpha$ is a trivial cocycle. This shows that for all non-trivial cocycles $\alpha$ of $Z^i$, $||\alpha|| \geq \eta_i$. Part (\ref{item:thm:Gromov-systolic:second}) of Theorem~\ref{thm:Gromov-systolic} is proven. In fact we prove along the way also a linear systolic lower bound for every $\alpha \in Z^2(Y,F_2)$, even though, this is not needed for Theorem~\ref{thm:Gromov-systolic}.

To prove part (\ref{item:thm:Gromov-systolic:first}) we argue as follows: Let $\beta \in B^{i+1}(X,\F_2) =B^{i+1}(Y,\F_2) $ ($i=0,1$). We have to show that there exists $\alpha \in C^i$ with $\delta_i(\alpha) = \beta$ and $||\alpha|| \leq \mu||\beta||$ for some $\mu > 0$. Arguing as before we can replace $\beta$ by a locally minimal cochain $\tilde{\beta}$ which is still in $B^{i+1}$,  $||\tilde{\beta}|| \leq ||\beta||$  and $\tilde{\beta} - \beta = \delta_i(\gamma)$ with $||\gamma|| \leq c_1 ||\beta||$ for some constant $c_1$, when $c_1$ depends neither on $Y$ nor on $\beta$. Note now that for every $\alpha \in C^i$, $||\alpha|| \leq 1$
Now, if $||\beta|| > \eta_i$, then there is nothing to prove, as we can take $\mu > \frac{1}{\eta_i}$. So assume $||\beta|| \leq \eta_i$; in this case $||\tilde{\beta}|| \leq \eta_i$ and hence by Theorem~\ref{thm:isoparametric}, $||\delta_{i+1}(\tilde{\beta})|| \geq \epsilon_i||\tilde{\beta}||$. But, $\tilde{\beta} \in B^{i+1}$, hence $\delta_{i+1}(\tilde{\beta}) =0$. This, implies that $\tilde{\beta} =0$. We saw that $\tilde{\beta} - \beta = \delta_i(\gamma)$ with $||\gamma|| \leq c_1||\beta||$, so now (\ref{item:thm:Gromov-systolic:first}) of Theorem~\ref{thm:Gromov-systolic} is also verified. Theorem~\ref{thm:main-thm} now follows.

\section{Isoperimetric inequalities - the $2$-dimensional case}\label{section:isoparametric-proof}
What is left in order to deduce the main theorem is to prove Theorem~\ref{thm:isoparametric} - the isoperimetric inequalities for the $3$-dimensional Ramanujan complexes. These proofs are somewhat long, technical and complicated. We keep these proofs for the full version of the paper~\cite{KaufmanKazhdanLubotzky}. What we give here is a complete proof of a baby version of Theorem~\ref{thm:isoparametric}, namely, the isoperimetric inequalities for Ramanujan complexes of dimension $2$. This will illustrate the main idea of the proof of Theorem~\ref{thm:isoparametric}, saving a lot of the technical difficulties. At the end of the proof for the $2$-dimensional case, we will explain briefly the challenges arising in dimension $3$. So here we prove the following:

\begin{thm}\label{thm:isoparametric-dim-2}
Let $q \gg 0$ (i.e. $q>q_0=q_0(2)$), $F=\F_q((t))$, $\calB=\calB_3(F)$ the $2$-dimensional Bruhat-Tits building associated with $PGL_3(F)$ and $Y$ a non-partite Ramanujan quotient of $\calB$. Then there exist $\eta_0,\eta_1,\epsilon_0,\epsilon_1$, all greater than zero such that:
For $\alpha \in C^i(Y,\F_2)$ a locally minimal cochain with $||\alpha|| \leq \eta_i$, $||\delta_i(\alpha)|| \geq \epsilon_i||\alpha||$.
\end{thm}

{\bf Remark:} We will see below that, in fact, for $i=0$, we do not need the assumption $||\alpha|| \leq \eta_0$ (or we can take $\eta_0$ =1 as in this case $||\alpha||$ is always at most $\frac{1}{2}$). If we would prove Theorem~\ref{thm:isoparametric-dim-2} also for $i=1$ without the assumption $||\alpha|| \leq \eta_1$, then it would follow that $Y$ itself is a coboundary expander and hence also a topological expander. But, as shown in~\cite{KaufmanKazhdanLubotzky}, for some $Y$'s, $H^1(Y,\F_2) \neq 0$. This implies that in general Theorem~\ref{thm:isoparametric-dim-2} is not true without the assumption that $||\alpha|| \leq \eta_1$. (Indeed, take $\alpha$ to be a locally minimal representative of a non-trivial cocycle in $Z^1(Y,\F_2)$, then $\delta_1(\alpha) = 0$). It is still very much plausible that $Y$ is a topological expander, but this is still open. This is the reason that we had to move to the $2$-skeleton of the $3$-dimensional Ramanujan complex.

We now prove Theorem~\ref{thm:isoparametric-dim-2}. We recall the properties of $Y$ we are using:
\begin{description}
\item[($A_2$)]  The $1$-skeleton $Y^{(1)}$ of $Y$ is a $k$-regular graph of degree $Q={3 \choose 1}_q + {3 \choose 2}_q  = 2(q^2+q+1)$, and every eigenvalue $\lambda$ of $A_{Y^{(1)}}$ is either $Q$ or $|\lambda| \leq 6q$.
\item[($B_2$)] The link $Y_v$ of a vertex is the flag complex of $\F_q^3$. This is the bipartite graph with $Q=2(q^2+q+1)$ points of degree $q+1$ of the "lines versus points" of the projective plane $\P^2(\F_q)$, so every edge of $Y$ lies in $q+1$ triangles. The non trivial eigenvalues of $A_{Y_v}$ are $\pm\sqrt{q}$.
\end{description}

The proof of the first part of the theorem, i.e., $i=0$, is the same as the (by now standard) argument that quotients of a group with property $T$ are expanders. We skip this part. Again, for full details see~\cite{KaufmanKazhdanLubotzky}. We concentrate on $i=1$, which is the main novelty of the current paper. To this end, fix $\epsilon' > 0$ and assume $\alpha \in C^1(Y,\F_2)$, with $||\alpha|| \leq \frac{1}{4(1+\epsilon')}$, i.e., $|\alpha| \leq \frac{Qn}{8(1+\epsilon')}$, when $n=|Y(0)|$.

We first recall some properties of (one dimensional) expander graphs that we will be using.
Let $X=(V,E)$ be a finite connected graph, $A=A_X$ its adjacency matrix and $\Delta$ its laplacian, i.e., $\Delta:L^2(X) \rightarrow L^2(X)$ defined by
$\Delta(f)(v)=deg(v)f(v)-\sum_{y \thicksim v}f(y)$ where the sum is over the neighbors of $v$. If $X$ is $k$-regular then $\Delta=kI-A$.
It is well known that the eigenvalues of $\Delta$ (and $A$) are intimately connected with the expansion properties of $X$. We will use the following result due to Alon and Milman~\cite[Prop 4.2.5]{LubotzkyBook}.

\begin{proposition}~\label{prop-cheeger} Let $\lambda=\lambda_1(X)$ be the smallest positive eigenvalue of $\Delta$.
\begin{enumerate}
\item For every subset $W \subseteq V$,
$$ |E(W,\bar{W})| \geq \frac{|W||\bar{W}|}{|V|} \lambda_1(X),$$ where $E(W,\bar{W})$ denotes the set of edges from $W$ to its complement $\bar{W}$.

\item The Cheeger constant $h(X)$ satisfies:
$$h(X):=\mbox{min}_{W \subseteq V} \frac{|E(W,\bar{W})|}{\mbox{min}(|W|,|\bar{W}|)} \geq \frac{\lambda_1(X)}{2}.$$

\item If $X$ is $k$-regular then $E(W):=E(W,W)$ satisfies:
$$ E(W)=\frac{1}{2}(k|W|-E(W,\bar{W})) \leq \frac{1}{2}(k-\frac{\bar{W}}{|V|}\lambda_1(X))|W|.$$
\end{enumerate}
\end{proposition}

\begin{lemma}~\label{lemma-traingles-counting-in-dim-two}
For $i=0,1,2,3$ denote by $t_i$, the number of triangles of $Y$ which contain exactly $i$ edges from $\alpha$. Then,
\begin{enumerate}
\item~\label{item-one-lemma-traingles-counting-in-dim-two} $t_1+2t_2+3t_3=(q+1)|\alpha|$.
\item~\label{item-two-lemma-traingles-counting-in-dim-two} $|\delta_1(\alpha)| = t_1+t_3$.
\item~\label{item-three-lemma-traingles-counting-in-dim-two} $\sum_{v \in Y(0)}|E_{Y_v}(\alpha_v, \overline{\alpha_v})|=2t_1+2t_2$.
\end{enumerate}
Here we consider $\alpha_v$, which is the set of edges in $\alpha$ touching $v$, as a set of vertices of the link $Y_v$. By $\overline{\alpha_v}$ we denote its complement there and $E_{Y_v}(\alpha_v, \overline{\alpha_v})$ is the set of edges from $\alpha_v$ to $\overline{\alpha_v}$.
\end{lemma}

\begin{proof} For (\ref{item-one-lemma-traingles-counting-in-dim-two}) we recall that every edge lies on $q+1$ triangles and a triangle which contributes to $t_i$ contains $i$ edges from $\alpha$. Part (\ref{item-two-lemma-traingles-counting-in-dim-two}) is simply the definition of $\delta_1(\alpha)$, which is the set of all triangles containing an odd number of edges from $\alpha$. For (\ref{item-three-lemma-traingles-counting-in-dim-two}) we argue as follows.

If $\triangle =\{v_0,v_1,v_2\}$ is a triangle of $Y$, then it contributes an edge at $Y_{v_k}$ ($\{ v_k\} = \{v_i,v_j,v_k\}\backslash \{v_i,v_j\} )$. This is the edge between $e_{i,k}=(v_i,v_k)$ and $e_{j,k}=(v_j,v_k)$ when we consider $e_{i,k}$ and $e_{j,k}$ as vertices of $Y_{v_k}$. This edge will be in $E_{Y_{v_k}}(\alpha_{v_k}, \overline{\alpha_{v_k}})$ if and only if exactly one of $\{e_{i,k},e_{j,k}\}$ is in $\alpha$. A case by case analysis of the four possibilities shows that if $\triangle$ has either $0$ or $3$ edges from $\alpha$ then $\triangle$ does not contribute anything to the left hand sum. On the other hand, if it has either $1$ or $2$ edges, it contributes $2$ to the sum. This proves the lemma.
\end{proof}

Fix now a small $\epsilon>0$ to be determined later and define:

\begin{definition} A vertex $v$ of $Y$ is called {\em thin} with respect to $\alpha$ if $|\alpha_v| < (1-\epsilon)\frac{Q}{2}$ and {\em thick} otherwise (note that by our local minimality assumption, $|\alpha_v| \leq \frac{Q}{2}$ for every $v$). Denote
$W=\{v \in V=Y(0) | \mbox{ } \exists e \in \alpha \mbox{ with } v \in e\}$, $R = \{v \in W | \mbox{ } v \mbox{ thin} \}$, $S = \{v \in W | \mbox{ } v \mbox{ thick} \}$.% = W \setminus R $.
\end{definition}
Let $r=\sum_{v \in R}|\alpha_v|$ and $s=\sum_{v \in S}|\alpha_v|$. As every edge in $\alpha$ contributes $2$ to $r+s$ we get the following:

%Denote
%\begin{itemize}
%\item $W=\{v \in V=Y(0) | \mbox{ } \exists e \in \alpha \mbox{ with } v \in e\}$.
%\item $R = \{v \in W | \mbox{ } v \mbox{ thin} \}$.
%\item $S = \{v \in W | \mbox{ } v \mbox{ thick} \} = W \setminus R $.
%\end{itemize}

\begin{lemma}~\label{lemma-$r+s$}
$r+s = 2|\alpha|$
\end{lemma}
%\begin{proof} Every edge in $\alpha$ contributes $2$ to the left hand side.
%\end{proof}

\begin{lemma}~\label{lemma-edges-exiting-alpha-v}
\begin{enumerate}
\item~\label{item-one-lemma-edges-exiting-alpha-v} For every $v \in V$, $|E_{Y_v}(\alpha_v, \overline{\alpha_v})| \geq \frac{1}{2}(q+1-\sqrt{q})|\alpha_v|$.
\item~\label{item-two-lemma-edges-exiting-alpha-v} If $v$ is thin, then $|E_{Y_v}(\alpha_v, \overline{\alpha_v})| \geq \frac{(1+\epsilon)}{2}(q+1-\sqrt{q})|\alpha_v|$.
\end{enumerate}
\end{lemma}

\begin{proof} As mentioned in $(B_2)$, the link $Y_v$ is the "line versus points" graph of the projective plane. It is a $(q+1)$-regular graph whose eigenvalues are $\pm(q+1)$ and $\pm \sqrt{q}$. Hence,
$\lambda_1(Y_v) = (q+1)-\sqrt{q}$. Part~\ref{item-one-lemma-edges-exiting-alpha-v} now follows from Proposition~\ref{prop-cheeger}, and similarly
part~\ref{item-two-lemma-edges-exiting-alpha-v}.
\end{proof}

We can deduce

\begin{lemma}~\label{lemma-2t-1+2t-2}
$2t_1 + 2t_2 \geq (q+1-\sqrt{q})|\alpha| + \frac{\epsilon}{2} (q+1-\sqrt{q}) r$.
\end{lemma}

\begin{proof}
$2t_1 + 2t_2 = \sum_{v \in W}E_{Y_v}(\alpha_v, \overline{\alpha_v})$.
The last equals to the following:
\begin{eqnarray}
= &   \sum_{v \in R}E_{Y_v}(\alpha_v, \overline{\alpha_v}) + \sum_{v \in S}E_{Y_v}(\alpha_v, \overline{\alpha_v}) \\
%\geq & \sum_{v \in R}\frac{(1+\epsilon)}{2}(q+1-\sqrt{q})|\alpha_v| \\
%+ & \sum_{v \in S}\frac{1}{2}(q+1-\sqrt{q})|\alpha_v| \\
\geq & \frac{(1+\epsilon)}{2}(q+1-\sqrt{q})r + \frac{1}{2}(q+1-\sqrt{q})s \\
= & \frac{1}{2}(q+1-\sqrt{q})(r+s)+\frac{\epsilon}{2} (q+1-\sqrt{q})r \\
= &  (q+1-\sqrt{q})|\alpha| + \frac{\epsilon}{2}(q+1-\sqrt{q})r
\end{eqnarray}
In the first equation we have used Lemma~\ref{lemma-traingles-counting-in-dim-two}, part (\ref{item-three-lemma-traingles-counting-in-dim-two}) and in the last one Lemma~\ref{lemma-$r+s$}. The inequality follows from Lemma~\ref{lemma-edges-exiting-alpha-v}.
\end{proof}
%\tnote{check, should it be epsilon divided by two?}

\begin{lemma}\label{lemma-t-1-minus-3t-3}
$t_1 - 3t_3 \geq \frac{\epsilon}{2} (q+1-\sqrt{q})r - \sqrt{q}|\alpha|$.
\end{lemma}

\begin{proof} Subtract equation (\ref{item-one-lemma-traingles-counting-in-dim-two}) in Lemma~\ref{lemma-traingles-counting-in-dim-two} form the equation obtained in Lemma~\ref{lemma-2t-1+2t-2}.
\end{proof}

Our goal now is to show that $r$, the contribution of the thin vertices, is greater than $c' \cdot|\alpha|$, where $c'$ does not depend on $Y$ (actually, it does not even depend on $q$). This will prove that for $q$ large enough $t_1 \geq cq|\alpha|$ and the theorem will follow. Up to now we have used only the local structure of $Y$, the links. Now we will use the global structure, the fact that its $1$-skeleton is almost a Ramanujan graph.

\begin{lemma}\label{lemma-bound-on-number-of-edges-within-thick-vertices} The total number of edges in $Y^{(1)}$ between the thick vertices is bounded as follows:
$$|E_{Y^{(1)}}(S)| \leq |\alpha|(\frac{1}{(1-\epsilon)^2(1+\epsilon')} + \frac{12q}{(1-\epsilon)Q}). $$
\end{lemma}

\begin{proof}
Recall, that by $(B_2)$, the second largest eigenvalue of the adjacency matrix of $Y^{(1)}$ is bounded from above by $6q$.
So $\lambda_1(Y^{(1)}) \geq Q-6q=2{q^2}-4q+1$. Note now that every vertex in $S$ touches at least $(1-\epsilon)\frac{Q}{2}$ edges of $\alpha$, hence $|S| \leq \frac{2|\alpha|}{(1-\epsilon)\frac{Q}{2}}=\frac{4|\alpha|}{(1-\epsilon)Q}$. Proposition~\ref{prop-cheeger} implies therefore (when $n:=|Y(0)|$)

\begin{eqnarray}
|E(S)| & \leq & \frac{1}{2}(Q-\frac{|\overline{S}|}{n}\lambda_1(Y^{(1)}))|S| \\
     & \leq & \frac{1}{2}(Q-\frac{|\overline{S}|}{n}(Q-6q))|S| \\
     & =  & \frac{1}{2}(Q(1-\frac{|\overline{S}|}{n})+6q \frac{|\overline{S}|}{n})|S|\\
     & \leq & \frac{1}{2}(Q \frac{|S|}{n}+6q)|S| \\
     & \leq & \frac{1}{2}(\frac{4|\alpha|}{(1-\epsilon)n}+6q)|S|
\end{eqnarray}

Recall, that $|\alpha| \leq \frac{Qn}{8(1+\epsilon')}$ and hence, $|E(S)| \leq (\frac{2Q}{8(1-\epsilon)(1+\epsilon')}+3q)\frac{4|\alpha|}{(1-\epsilon)Q}=|\alpha|(\frac{1}{(1-\epsilon)^2(1+\epsilon')} + \frac{12q}{(1-\epsilon)Q})$.
\end{proof}

\begin{proof}(of Theorem~\ref{thm:isoparametric-dim-2})
We can now finish the proof of Theorem~\ref{thm:isoparametric-dim-2}. Choose $\epsilon > 0$ such that $\frac{1}{(1-\epsilon)^2(1+\epsilon')} < 1$
and then assume that $q$ is sufficiently large such that $\frac{1}{(1-\epsilon)^2(1+\epsilon')} + \frac{12q}{(1-\epsilon)Q} < 1-\xi < 1$, for some $\xi > 0$. This now means by Lemma~\ref{lemma-bound-on-number-of-edges-within-thick-vertices} that at most $(1-\xi)$ of the edges in $\alpha$ are between two thick vertices, namely, for at least $\xi|\alpha|$ edges, one of their endpoints is thin. This implies that $r \geq \xi|\alpha|$. Plugging this in Lemma~\ref{lemma-t-1-minus-3t-3}, we get $t_1  \geq  (\frac{\epsilon}{2}(q+1-\sqrt{q})\xi- \sqrt{q})|\alpha|$. Again, if $q$ is large enough this means that $|\delta_1(\alpha)| \geq t_1 \geq \epsilon_1 q |\alpha|$ and Theorem~\ref{thm:isoparametric-dim-2} is proved with $\eta_1=\frac{1}{4(1+\epsilon')}$.
\end{proof}

Let us mention that along the way we have proved two facts which are worth formulating separately.

\begin{cor}\label{cor-from-2-dim-result} In the notations and assumptions as above. For every $\epsilon' > 0$, if $q \geq q(\epsilon') \gg 0$, then we have:
\begin{enumerate}
\item If $\alpha \in B^1(X,\F_2)$ is a locally minimal coboundary with $|\alpha| < \frac{1}{4(1+\epsilon')}|X(1)|$ then $\alpha = 0$.
\item If $\alpha \in Z^1(X,\F_2) \setminus B^1(X,\F_2)$, then $|\alpha| > \frac{1}{4(1+\epsilon')}|X(1)|$. In particular, every representative of a non-trivial cohomology class has linear size support.
\end{enumerate}
%in particular, $q \gg 0$: There exists $\rho > 0$ such that:
%\begin{enumerate}
%\item If $\alpha \in B^1(X,\F_2)$ is a locally minimal coboundary with $|\alpha| < \frac{1}{4(1+\epsilon')}|X(1)|$ $\alpha = 0$.
%\item If $\alpha \in Z^1(X,\F_2) \setminus B^1(X,\F_2)$, then $|\alpha| > \frac{1}{4(1+\epsilon')}|X(1)|$. In particular, every representative of a non-trivial cohomology class has linear size support.
%\end{enumerate}
\end{cor}

This is a systolic inequality. This is actually a very strong systolic lower bound as any class of $Z^1$ has a representative $\alpha$ with $|\alpha| \leq \frac{|X(1)|}{2}$, and $\epsilon'$ can be chosen as small as we wish.
%(paying a price in $q \gg 0$ (i.e. $q>q_0=q_0(d)$)).
Note that as shown in~\cite{KaufmanKazhdanLubotzky}, there are indeed cases that $H^1(X,\F_2) \neq \{0\}$, so the second item of Corollary~\ref{cor-from-2-dim-result} is a non-vacuum systolic result. Such results are of potential interest for quantum error-correcting codes (see \cite{Zemor},\cite{LubotzkyGuth} and the references therein).

%Let us note that along the way we have proved here a strong $\F_2$-systolic inequality. Such results are of potential applications for quantum %error-correcting codes (see \cite{Zemor},\cite{LubotzkyGuth} and the references therein).

We finish by outlining the additional difficulties in proving Theorem~\ref{thm:isoparametric}. Some of the difficulties are technical, e.g., the difference between $(B_2)$ and $(B_3)$ where the links are more complicated. A more essential difficulty is with the notion of "thinness": the appropriate notion of thinness for vertices and edges in the high dimensional case is far from being obvious. See~\cite{KaufmanKazhdanLubotzky} for more.
%for edges in the $3$-dimensional case the definition is analogous to the one for vertices in the $2$-dimensional case. But, the main difficulty is with 

\remove{
We end up with saying that the proof of Theorem~\ref{thm:isoparametric} is different than the proof of Theorem~\ref{thm:isoparametric-dim-2} mainly because of the difference between $(B_2)$ in this section and $(B_3)$ in section~\ref{section:isoparametric}, where the link is a $3$-partite and therefore induces $3={3 \choose 2}$ bipartite graphs, each has to be handled separately. Also, working with $2$-cochains instead of $1$-cochains put more challenges. For example, when $\mbox{dim}Y=2$, $Y_v$ is the "lines versus points" graph whose expansion properties are easy to handle. But, when $\mbox{dim}Y=3$, $Y_v$ is a $2$-dimensional spherical building and for its expansion properties we need to refer to~\cite{Gromov} and~\cite{LubotzkyMeshulamMozes}. Still, the basic idea is the same, we define thin and thick vertices and also thin and thick edges and analyze their contributions. This is done by local consideration (by detailed study of the links). Then, we have to prove some non-trivial lower bounds on the total contribution of the thin vertices. This is where the global structure of $Y^{(1)}$ comes into the game and where the assumption that $||\alpha|| \leq \eta$ has to be assumed. Full details can be found in~\cite{KaufmanKazhdanLubotzky}.
}

\section{Acknowledgments}
The authors are grateful to G. Kalai, R. Meshulam, E. Mossel, S. Mozes, A. Rapinchuk, J. Solomon and U. Wagner for useful discussions and advice. We thank also the ERC, ISF, BSF and NSF for their support.

\end{document}